\documentclass[reqno]{amsart}

\usepackage{bm}
\usepackage{amssymb}
\usepackage{ifthen}
\usepackage{amsthm}
\usepackage{microtype}
\usepackage{pgfplots}
\usetikzlibrary{pgfplots.groupplots}
\pgfplotsset{compat=1.5}
\usepackage{tikz}
\usetikzlibrary{external}
\usepackage{xparse}
\usepackage{framed}
\usepackage{theoremref}
\usepackage{comment}
\usepackage{subfigure}
\usepackage{float}
\usepackage{rotating}
\usepackage{changepage}
\usepackage[showframe=false, margin = 1.5     in]{geometry}
\usepackage{paralist}
\usepackage{enumitem}

\usepackage{amsmath,amsthm,amssymb,bm,graphicx}
\newtheorem{theorem}{Theorem}
\newtheorem{lemma}[theorem]{Lemma}

\newtheorem{proposition}[theorem]{Proposition}

\theoremstyle{plain}
\usepackage{graphicx}
\graphicspath{ {images/} }

\DeclareMathOperator{\Ber}{Ber}
\DeclareMathOperator{\Bin}{Bin}
\DeclareMathOperator{\Geo}{Geo}
\DeclareMathOperator{\E}{\mathbf{E}}
\newcommand{\f}{\frac}
\renewcommand{\P}{\mathbf P}
\newcommand{\ak}{\lceil \alpha k \rceil}

\title{Frog model wakeup time on the complete graph}

\author[N. Cartern]{Nikki Cartern}
\address{University of Portland}
\email{cartern16@up.edu}

\author[B. Dygert]{Brittany Dygert}
\address{Seattle Pacific University}
\email{bddygert@gmail.com}

\author[M. Junge]{Matthew Junge}
\address{University of Washington}
\email{jungem@math.wasington.edu}

\author[S. Lacina]{Stephen Lacina}
\address{University of Oklahoma}
\email{stephen.d.lacina-1@ou.edu}

\author[C. Litterell]{Collin Litterell}
\address{University of Washington}
\email{collin.litterell@gmail.com}

\author[A. Stromme]{Austin Stromme}
\address{University of Washington}
\email{austinjstromme@gmail.com}

\author[A. You]{Andrew You}
\address{Duke University}
\email{andrew.kh.you@gmail.com}

\begin{document}
\maketitle

\begin{abstract}

The frog model is a system of random walks where active particles set sleeping particles in motion. On the complete graph with $n$ vertices it is equivalent to a well-understood rumor spreading model. We given an alternate and elementary proof that the wake-up time, i.e.\ the expected time for every particle to be activated, is $\Theta(\log n)$. Additionally, we give an explicit distributional equation for the wakeup time as a weighted sum of geometric random variables. This project was part of the University of Washington Research Experience for Undergraduates program.

\end{abstract}

\section{Introduction}

The \emph{frog model} starts with an awake frog at the root of a graph and one sleeping frog at each other vertex. In discrete time, awake frogs perform nearest neighbor simple random walks and wake any sleeping frogs they encounter. When first introduced by K. Ravishankar about twenty years ago, the model was known as the egg-model (see \cite{telcs1999}); R. Durrett is
credited with the zoomorphism of viewing particles as frogs. This
likely comes from the chaotic way particles wake up. We study the model on the complete graph with $n$ vertices, $K_n$. In particular we deduce that the \emph{wakeup time} $T_n$, the time for all frogs to wake up, has expected value on the order of $\log n$.

It was brought to our attention in the final stages of this project 
 that the wakeup time for the frog model on $K_n$ is equivalent to a rumor spreading model introduced in \cite{original}. This model starts with a town of $n$ people where one knows a rumor. At each time step those who know the rumor call a uniformly random resident and inform them.  The frog model can be naturally coupled with the spread of the rumor so that the number of awake frogs is the same as the number of informed residents. Hence the wakeup time is equivalent to all $n$ residents knowing the rumor. We remark that the locations of the frogs are an extra bit of randomness not accounted for in the rumor spreading model. For this reason, the coupling only works on $K_n$.

In \cite{original} they show that $T_n/ \log_2 n \overset{\mathbf P} \to 1 + \log n + O(1)$ which implies our theorem. 
The idea of the proof in \cite{original} is to break up the spread of the rumor into five stages (see the appendix for a list of the stages). For example, phase one is the time to wake up $N$ frogs with $N$ some large fixed constant. A finer analysis in \cite{second} shows that $T_n = \log_2 n + \log n + O(1)$ in probability. They show that waking is closely approximated by a deterministic equation (see the appendix). A tight analysis for $\E T_n$ is in \cite{third}, where they use three phases and some sophisticated estimates to show that $\E T_n = c \log n + b + o(1)$ with $c = 1+ (1/ \log(2))\approx 2.44$ and $b= 2.765$. Our result is less precise, but the proof is more elementary. We use two phases and only rely on couplings and Markov's inequality.

Although it is equivalent to the spread of a rumor on $K_n$, the wakeup time for the frog model has otherwise not been studied.
 The recent article \cite[Open Question 5]{poisson} introduces the problem on finite $d$-ary trees. A survey article \cite{frogs} asks a similar question for a variant of the frog model where frogs perish after taking $t$ steps. They propose a study of the minimal $t$ that guarantees at least half the frogs on a given graph will be activated with probability greater than $1/2$.
It is claimed (without proof) that this value on the complete graph is $O(\log n)$. Though not equivalent, this is closely related to our result for the frog model cover time.

On finite graphs the frog model is a model for epidemics, or the spread of a rumor. The article \cite{Coop09} describes many related variants. It also appears in physics literature as a model for combustion \cite{combustion} known as $A + B \mapsto 2A$, where we replace awake and sleeping frogs with flames and fuel, respectively. The combustion process is studied on $\mathbb Z^d$. Noteworthy theorems include the fact that the origin is visited infinitely often for all $d \geq 1$ \cite{telcs1999} and a shape theorem that says, when properly scaled, the set of activated vertices converges to a convex region in the unit simplex (\cite{shape, random_shape}).

A model that is in a loose sense the reverse of the frog model is coalescing random walk. Introduced in \cite{erdos1974}, this is the process that starts with a particle at each site and when particles collide they coalesce into one. Like the frog model, this is typically studied in $\mathbb Z^d$. For instance, coalescing random walk is recurrent for all $d \geq 1$. This was first shown in \cite{first}, and refined further in \cite{griffeath, kesten_number, arratia, arratia2}. Computer science literature studies coalescing random walk on finite graphs. Of particular interest is the \emph{coalescance time}; the expected time for all particles to coalesce into a single particle. In \cite{cox} they study this on the torus. The recent paper \cite{cs} gives bounds on rather general graphs.

Our result should be compared with the cover time for multiple random walks on a graph (\cite{Elsässer20112623} and \cite{Alon07})). The basic question is how the cover time is reduced by using the combined ranges of $k$ random walks. This is studied on a variety of different graphs, and the speedup depends on the graph structure. For the complete graph, \cite{Alon07} cites the folklore (we give a proof in \thref{lem:couplings} \ref{loop}) that the speedup is linear, meaning the cover time for a single random walk is $k$ times the cover time for $k$ random walkers. All of the results for speedup of multiple walk take the worst-case scenario across every starting configuration for the $k$ walkers. The frog model is different in that we have only one possible starting configuration, and just one active particle. However, the placement of sleeping frogs is optimal in the sense that activated particles are more likely to be near unexplored sites. We ask a question regarding this in Further Questions (i).

\subsection{Main theorem and overview} \label{sec:main}

Before stating the theorem we review \emph{asymptotic notation}. We say that $f(n) = O(g(n))$ if there exists $c >0$ such that for all sufficiently large $n$ it holds that $f(n) \leq c g(n).$ We write $f(n) = \Theta(g(n))$ if $f(n) = O(g(n))$ and $g(n) = O(f(n))$. The wakeup time, $T_n$, is formally defined in the next section. 

\begin{theorem} \thlabel{thm:main}

$\E T_n = \Theta(\log(n))$.

\end{theorem}

Additionally, in Section \ref{sec:formula} we give an explicit recursive formula for the distribution of $T_n$. This is in \thref{prop:formula}. The formula involves some sophisticated combinatorial objects and, combined with the formula in \cite{third} for $\E T_n$, yields a bound on their growth that could be of independent interest.

Most of our work is done on $K_n^\circ$, the complete graph with a self loop at each vertex. In \thref{lem:couplings} \ref{loop} we show the frog model on $K_n^\circ$ has a stochastically larger wakeup time.
We then prove \thref{thm:main} in two phases. First, we show
that it takes logarithmic time to wake the first $n/2$ frogs. This is done by embedding a process that grows slower than the frog model, but still (on average) grows exponentially. The idea is to, when say $k$ frogs are awake, only allow more frogs to wake up if at least $\alpha k$ asleep frogs are visited. This occurs with some probability $q_{k,n}$. We show in \thref{prop:p*} that $$\inf_{n \geq 3}\left( \min_{k < n/2} q_{k,n} \right) \geq  p_* >0.$$ This lower bound is obtained by having the $k$ awake frogs jump one at a time, and thinking of the number of single jumps to wake the $(i+1)$st frog as a Geometric with mean $\f{n-k-i}{n}$ waiting time. An application of Markov's inequality gives a uniform bound in terms of $\alpha$ for all $k< n/2$ and $n\geq 3$.

Next, thinking of each time an $\alpha$ proportion wakes up as a Bernoulli($q_{k,n}$) trial, the number of successes after $t$ steps is stochastically larger than a $\Bin(t,p_*)$ random variable.
Thus, the number of frogs awake at time $t$ is stochastically larger than
$$(1 + \alpha)^{\Bin(t,p_*)}.$$
Moreover, the time for this quantity to exceed $n/2$ is a sum of $O(\log n )$ geometric random variables with mean $p_*$. This is made formal in \thref{lem:couplings} \ref{key}. We can conclude that the expected time it takes to wake the first $n/2$ frogs is $O(\log n)$.

Once half the frogs are awake, we ignore the contribution of any new frogs added and show that $n/2$ frogs cover the remaining vertices in $O(\log n)$ steps. This is made precise in \thref{lem:couplings} \ref{loop} by reducing to the coupon collector problem.

\subsection{Further Questions}
The wakeup time for the frog model is a largely unexplored topic. There are many further questions one could ask.  We remark that \cite{frogs} and \cite{poisson} discuss a few other problems on finite graphs.


\begin{enumerate} [label = {(\roman*)}]

\item {It is interesting to compare the wakeup time for the frog model on $G$, a graph with $n$ vertices, to the cover time for $n$ independent random walks on $G$ started in the least optimal starting configuration. Perhaps the frogs being evenly spread might overcome the disadvantage of starting with only one awake particle.} \emph{Are there graphs for which the frog model wakeup time is faster than the cover time for $n$-multiple random walks?} {The full binary tree of height $n$ is a good candidate. The expected cover time with $2^{n}$ particles started at the same leaf is $O(n^2 (2/\sqrt 2)^n)$ (see \cite{bincov}), whereas the wakeup time is conjectured to be polynomial in $n$ (see \cite{poisson}).}

\item Let $G(n,p)$ denote an Erd{\H o}s-R{\'e}nyi graph (i.e. the random graph obtained by keeping each edge in $K_n$ with probability $p$). \emph{What is the wakeup time for the frog model on $G(n,p)$?}{ For fixed $p>0$, this should still with high probability be $O(\log n)$, but for $p_n$ decaying with $n$ the graph is sparser and the wakeup time is possibly larger.}

\item  \emph{What is the wakeup time for other graphs?} {For instance, the path, cycle, and grid.}


\end{enumerate}


\section{Formal model and couplings}

Here we give a formal definition of the frog model. Then we describe in \thref{lem:couplings} the five couplings we depend on in proving our main theorem.

\subsection{Formal definition of the frog model}

We borrow much of our notation from \cite{shape}. Let $V$ be the vertex set of $K_n$. Consider the collection, $\{F_v(t) \colon v \in V\}$, of independent random walks on $K_n$ each satisfying $F_v(0) = v$. These random walks correspond to the the trajectory of each frog. We now introduce stopping times to account for the waking up that occurs. Define
$$t(v,u) = \min_t \{F_v(t) = u\},$$
the time that the frog originally at vertex $v$ takes to reach vertex $u$. Also define
$$T(v,u) = \inf \left\{\sum_{i=1}^k t(v_{i-1},v_i) \colon v_0=v, \ldots, v_k = u \text{ for some } k\right\},$$
the first passage time from $v$ to $u$ in the frog model. Then $T(v_0,u)$ gives the time it takes for $u$ to be woken. For each $u \in V$, the position of the frog originally at $u$ at time $t$ is defined to be
$$P_u(t) = \begin{cases} u, &t \leq T(v_0,u) \\ F_u(t-T(v_0,u)), &t > T(v_0,u) \end{cases}.$$
With this we can define $\Lambda(t) = \{u \in V \colon T(v_0,u) \leq t\}$, the set of sites that have been visited by time $t$ or the set of awake frogs at time $t$. Define the number of frogs awake at time $t$ to be $N_t = |\Lambda(t)|$. Thus, the time to wake all of the frogs is $T_n = \inf \{t \colon N_t = n\}.$

\subsection{Couplings and stochastic dominance} \label{sec:dominance}

The frog model only depends on the underlying random walk trajectories. It has the nice feature that restricting the range of frogs, or ignoring woken frogs yields models with monotonically slower waking behavior. This is made formal using couplings and stochastic dominance.

Let $X$ and $Y$ be two random variables. If for each $a >0$  we have $\mathbf{P}[Y \geq a] \geq \mathbf{P}[X \geq a]$ then we say that $Y$ \emph{stochastically dominates} $X$, written $X\preceq Y$. A thorough reference on stochastic domination is \cite{SS}. Note that if $A \succeq B$, then $\mathbf{E} A \geq \mathbf{E}B$. An equivalent condition to stochastic dominance is that
$X\preceq Y$ if and only if there exists a coupling $(X,Y)$ with $X\leq Y$ a.s. Formally, a \emph{coupling} is a probability space with (possibly dependent) random variables $X'$ and $Y'$ that have the same distribution as $X$ and $Y$, respectively. Couplings can often be described intuitively and rigorously in words. In the following lemma we describe all of the couplings used in this paper.

\begin{lemma} \thlabel{lem:couplings} The following stochastic dominance relations hold:

    \begin{enumerate}[label = (\Roman*)]
        \item \label{binom} Let $\{q_i\}_{i=1}^t$ be a sequence in $[0,1]$ with $q_i > p$ for all $i=1,2,\hdots,, t$.  It holds that
        $$\sum_{i=1}^t \Ber(q_i) \succeq  \Bin(t,p).$$
        Here $\Ber(p)$ denotes a Bernoulli-$p$ random variable and $\Bin(t,p)$ is a binomial random variable with $t$ trials.

        \item \label{loop} Let $K_n^\circ$ be the complete graph with a self-loop added to each vertex. If $T_n^\circ$ is the wakeup time for the frog model on $K_n^\circ$, then
        $$T_n \preceq T_n^\circ.$$

        \item \label{breakdown} Let $\tau_{n/2}$ and $C_{n/2}$ be as defined in the proof of \thref{thm:main}. It holds that
        $$T_n \preceq \tau_{n/2} + C_{n/2}.$$

         \item \label{loop} Let $K_n^\circ$ be the complete graph with a self-loop added to each vertex. Define $C_k^\circ$ to be the time for $k$ random walks to collectively visit every vertex of $K_n^\circ$, and define $C_k$ analogously for $K_n$. It holds that
        \begin{align}
        C_{k} \preceq C_{k}^\circ \overset{d}= \frac {C_1^\circ} k \overset{d}= \frac 1 k \sum_{i=1}^n \Geo\left(\frac{n-i}{n}\right).\label{eq:Ck}
        \end{align}
Where $\Geo(p)$ is the number of Bernoulli-$p$ trials until a success occurs.
        \item \label{key} Consider a modified frog model on $K_n^\circ$ where, when $k$ frogs are active, more frogs wake up only if at least $\alpha k$ sleeping frogs are visited at the next step (for some fixed $\alpha >0$). When this occurs we select an arbitrary subset of $\lceil \alpha k\rceil $ of these frogs and allow them to wake up. The others remain asleep. Thus, on this event there are at least $(1+\alpha)k$ frogs awake. \thref{prop:p*} shows that waking at least $\alpha k$ frogs occurs with probability at least $p_* >0$ for any $k$ and and sufficiently large $n$. We then have
        $$(1 + \alpha)^{\Bin(t,p_*)} \preceq N_t,$$
        and for $\tau_{n/2} := \inf\{t \colon N_t \geq n/2\}$ and $n_*:= \left\lfloor{\frac{\log(n/2)}{\log(1+ \alpha)}} \right\rfloor$ we have
        \begin{align*}\tau_{n/2} &\preceq \inf\{ t \colon (1 + \alpha)^{\Bin(t,p_*)} \geq n/2 \}
\overset{d} = \sum_1^{n_*} \Geo(p_*).
\end{align*}
    \end{enumerate}

\end{lemma}

\begin{proof} All of the proofs establish stochastic dominance via couplings.
    \begin{enumerate}[label = (\Roman*)]
           \item 
           Define $X = \sum_{i=1}^t\Ber(q_i)$ and $Y = \sum_{i=1}^t \Ber(p) = \Bin(t,p)$. Let $\{U_i\}_{i=1}^t$ be uniform $[0,1]$ random variables, so that
           \begin{align*}
               X &\overset{d} = \sum_{i=1}^t 1\{U_i \leq q_i\}, \\
Y & \overset{d} = \sum_{i=1}^t 1\{U_i \leq p\}.
               \end{align*}
               Our hypothesis $q_i > p$ guarantees that $X\geq Y$ for all realizations of the $U_i$. 

           \item 
           Pair the frogs on $K_n$ and $K_n^\circ$ in the natural way. Whenever a frog on $K_n^\circ$ moves to a new vertex, have the corresponding frog on $K_n$ follow it. In this way, the frogs on each graph perform random walks, but those on $K_n^\circ$ possibly spend extra steps traveling self-loops. This coupling ensures that $T_n \leq T_n^\circ$ in every realization of the model. 

           \item 
           Run the frog model up to time $\tau_{n/2}$. Of the $N_{\tau_{n/2}}$ frogs awake choose a batch of $n/2$ of them. Now think of this batch as paired to another frog model in the same configuration as ours at time $\tau_{n/2}$. Our $n/2$-batch frogs follow their counterparts. The time, $C_{n/2}$, it takes for the batch to visit all $n$ vertices of $K_n$ is at least as large as the $T_n - \tau_{n/2}$ steps taken by the frog model they are coupled with. In this way, the model restricted to a batch spends $\tau_{n/2} + C_{n/2}$ steps, which is at least the $T_n$ steps taken by the frog model.

           \item A similar coupling as in (II) gives $C_k \preceq C_k^\circ$. Observe that on $K_n^\circ$ every site is accessible in one step. Thus, the set of sites visited by $k$ random walks has the same law as the range of a single random walk in $k$ steps. It follows that $kC_k^\circ \overset{d} = C_1^\circ$. The last equality
           $$C_1^\circ \overset{d} = \sum_{i=1}^n \Geo\left( \frac{ n-i }{n}\right)$$
           follows from the observation that the waiting time for a single random walk on $K_n^\circ$ to increase its range from $i$ to $i+1$ is the waiting time to have a success in a sequence of Bernoulli($(n-i)/n$) trials. This is distributed as a $\Geo( (n-i)/n)$. As increases in the range are independent and skip-free, the claimed formula follows.
           \item By (II) we preserve the dominance relation by working on $K_n^\circ$. Since each successful increase in the number of frogs is a Bernoulli trial with probability at least $p_*$ it follows from (I) that this further dominates the random quantity $(1+ \alpha)^{\text{Bin}(t ,p_*)}$. The fact that $N_t$ dominates this quantity is a straightforward consequence of the fact that we are only ignoring frogs. And, the fact that $\tau_{n/2}$ is less than the time for $(1+ \alpha)^{\text{Bin}(t ,p_*)}$ to exceed $n/2$ follows immediately from the relationship between $(1+ \alpha)^{\Bin(t,p_*)}$ and $N_t$.

Taking the $\log$ of each side, we have the stopping time $\inf\{ t \colon (1+ \alpha)^{\Bin(t,p_*)} \geq n/2\}$ is equivalent to the time for a $\Bin(t,p_*)$ random variable to exceed $\log(n/2)/\log(1+ \alpha).$ The claimed distributional equality is just the fact that $\Bin(t,p_*)$ is a skip-free process which increases independently at each increment after $\Geo(p_*)$ steps.
    \end{enumerate}
\end{proof}

\section{Proving \thref{thm:main}}

We start by elaborating a bit on \thref{lem:couplings} \ref{key}. Let $\alpha <1$ be a yet to be chosen parameter. Define the probabilities $q_{k,n} = q_{k,n}(\alpha)$ that the frog model on $K_n^\circ$ with $k$ frogs awake wakes at least $\alpha k$ frogs in one time step. Let $p_*= p_*(\alpha) = \inf_{n \geq 3}\left( \min_{ k < n/2} q_{k,n}\right)$. Our first goal is to establish that $p_*$ is bounded away from 0.

\begin{proposition} \thlabel{prop:p*}
For $\alpha =  1/10$ it holds that $p_* \geq 1/37$.
\end{proposition}

\begin{proof}
It is useful to decompose one time step in the frog model with $k$ frogs awake on $K_n^\circ$ into $k$ steps by a single random walk. Notice that the number of steps by the single random walk to visit the first sleeping frog is the waiting time for a success in a sequence of $\Ber(\f{n-k}{n})$ trials (i.e.\ a $\Geo(\f{n-k}{n})$ random variable). Similarly, once $i$ of the $n-k$ frogs have been woken the waiting time is $\Geo(\f{n-k-i}{n})$. Let $X = X(k,n,\alpha) = \sum_{i=0}^{\ak} \Geo\left ( \tfrac{ n-k -i }{n} \right) .$ This represents the number of the $k$ awake frogs that jump in order to wake $\ak$ more.
It follows that $q_{k,n} = 1- \P[ X > k].$ And, by Markov's inequality
\begin{align}
q_{k,n} \geq 1 - \f{ \E X }{k}. \label{eq:lb}
\end{align}
We can use linearity and the fact that the mean of a $\Geo(p)$ random variable is $1/p$ to estimate $\E X$:
\begin{align*}
\E X = \sum_{i=0}^{ \ak } \f{ n }{ n - k -i } &= n \sum_{i=n- \ak }^{ n } \f {1} { i }  \leq n \f{n - (n- \ak)}{n - \ak} =  \f{ n \ak }{ n- \ak }.
\end{align*}
Since $\ak \leq \alpha k +1$ we can bound $\E X /k$ by
\begin{align*} \f{\E X}{k} &\leq \f{ \alpha k  +1}{k} \f{ n }{n - \alpha k -1} \\
			&= \f{ \alpha n }{n - \alpha k -1}+\f 1 k  \f{ n }{n - \alpha k -1}.
\end{align*}
Note that $q_{1,n} =1-\f 1 n \geq \f 23$ since we assume $n \geq 3$. Thus, we can work with $k \geq 2$. Also by assumption $k$ is no larger than $n/2$. We then preserve the above bound by setting $k=n/2$ in the negative terms and $k=2$ in the $\f 1 k$ term. This results in 
$$q_{k,n} \geq 1- \f{ \alpha n }{n - \alpha n/2 -1} - \f{ n }{2(n - \alpha n/2 -1)} = 1 - \f{ \alpha + \f 12} { 1 - \f \alpha 2 - 1/n}.$$
As $n \geq 3$ we arrive at
$p_* \geq 1 - \dfrac{\alpha + \f 12 }{ \f 23 - \f \alpha 2 }.$
If we evaluate at $\alpha = 1/10$, then $p_* \geq 1/37$, which completes the proof.
\end{proof}

\subsection{Proof of \thref{thm:main}}
With \thref{lem:couplings} and \thref{prop:p*} we can prove our main theorem.

\begin{proof}[Proof of \thref{thm:main}]

Since the number of frogs can at most double at each step we have $T_n \geq \log_2(n)$. The lower bound immediately follows. As for the asymptotic upper bound, let $\tau_{n/2} = \inf\{ t \colon N_t \geq n/2\}$ be the time to wake at least $n/2$ frogs, and let $C_{n/2}$ be the time for $n/2$ walkers to visit every vertex of $K_n$ (taken to be the maximum such time over all possible starting configurations of walkers). \thref{lem:couplings} \ref{breakdown} describes a coupling where we ignore the benefit of waking more frogs after time $\tau_{n/2}$ to conclude
\begin{align}T_n \preceq \tau_{n/2} + C_{n/2} .\label{eq:Tn}
\end{align}
Here `$\preceq$' denotes stochastic domination, see Section \ref{sec:dominance} for the definition.
The couplings in \thref{lem:couplings} \ref{loop} and  \ref{key} along with \thref{prop:p*} imply that there exists $\alpha, p_*>0$ such that
\begin{align*}
\tau_{n/2} \preceq  \sum_{i=1}^{ \left \lfloor \frac{ \log( n/2)  }{ \log(1+ \alpha)} \right \rfloor } \Geo(p_*)
 \quad \text{ and } \quad
 C_{n/2} &\preceq \frac{2}{n} \sum_{i=1}^{n} \Geo\left( \frac{ n-i} {n} \right).
\end{align*}
Using the fact that the expectation of a Geometric($p$), random variable is $1/p$ we can take the expectation of both sides of \eqref{eq:Tn} to obtain
$$\E T_{n} \leq  (1/p_*) \left \lfloor \frac{ \log( n/2)  }{ \log(1+ \alpha)} \right \rfloor + \frac{2} n  \sum_{i=1}^{n-1} \frac{n}{n-i} = O(\log n) .$$
Note the first summand above is $O(\log n)$ because $p_*$ and $\alpha$ are positive. The second summand is $O(\log n)$ by canceling the factors of $1/n$ and $n$ then comparing to the harmonic numbers $\sum_{i=1}^n \frac 1 i \approx \log n$.
\end{proof}


\section{Exact distribution of $T_n$} \label{sec:formula}
Let $\sigma_k = \sigma_k(n)$ be the time to wake up all $n$ frogs on the complete graph given that there are $k$ frogs currently awake. So, $\sigma_1 = T_n$.

\begin{proposition} \thlabel{prop:pformula}
Let $$p_{j,k} = \mathbf P[ \text{$k$ awake frogs visit $j$ sleeping frogs on the next step}].$$ It holds that

$$p_{j,k} = \frac{1}{(n  -1)^k} \sum_{\ell = j}^k {k \choose \ell} {n - k \choose j} S_{j,k}^{\ell} (k - 1)^{k - \ell},$$
where $S_{j,k}^\ell$ the number of ways to distribute $\ell$ balls into $k$ boxes so that no box is empty. This is given by the formula
$$
S_{j,k}^{\ell} =j! \cdot S(k,l) = \sum_{k = 1}^{j } {j \choose  j - k} (-1)^{j - k} k^{\ell},
$$
where $S(k,\ell)$ is Stirling's number of the second kind.
\end{proposition}

\begin{proof}

Observe that to wake up $j$ more frogs, we can use between $j$ and all $k$ frogs, and distribute them onto $j$ unvisited vertices. So let $\ell$ range between $j$ and $k$.  There are ${k \choose \ell}$ ways to choose $\ell$ of the already awoken frogs to visit the $j$ new vertices. We can choose $j$ new vertices in ${n - k \choose j}$ ways. Once we have chosen the $\ell$ frogs and $j$ new vertices, we can distribute them in (by definition) $S_{j,k}^{\ell}$ ways. Finally the remaining $k - \ell$ frogs can go to any of the $k -1$ already visited vertices they are adjacent to, so they can move in $(k - 1)^{k - \ell}$ ways.  Also, since there are $k$ frogs and we are thinking of them as distinguishable, each one of these events happens with probability $1/(n - 1)^k$.  Thus we get, by summing over all $\ell$, that
$$
p_{j,k }=\frac{1}{(n  -1)^k} \sum_{\ell = j}^k {k \choose \ell} {n - k \choose j} S_{j,k}^{\ell} (k - 1)^{k - \ell}.
$$

We claim that for fixed $\ell$:
$$
S_{j,k}^{\ell} =\sum_{k = 1}^{j } {j \choose  j - k} (-1)^{j - k} k^{\ell}
$$
To see this we proceed by inclusion exclusion principle. Observe $j^{\ell}$ counts the number of ways to distribute $\ell$ distinguishable balls to $j$ distinguishable boxes with some boxes left empty possibly.  So to count the number of ways with no boxes left empty, we should subtract the number with at least one empty, which is ${j \choose j - 1} (j - 1)^{\ell}$ since we have $ j -1$ boxes that we will possibly place balls in, and have $j - 1$ choices for each of the balls. But now all the ways with exactly two boxes left empty have been added once and subtracted twice (since we counted them ${2 \choose 1}$ times in the subtraction), so we should add ${j \choose j -2}(j - 2)^{\ell}$ to count the number of ways with at least two left empty. Keep going in this fashion to get
$$
S_{j,k}^{\ell} = \sum_{ k =1}^{j} {j \choose k} (-1)^{j - k} (j - k)^{\ell} = \sum_{ k =1}^j {j \choose j - k} (-1)^{j- k} (j - k)^{\ell}.
$$

\end{proof}

An explicit formula for time to wake all $n$ frogs on the complete graph with $n$ vertices is defined recursively as follows: \\*

\begin{proposition} \thlabel{prop:formula}
Let $p_{j,k}$ be as in \thref{prop:pformula} and let $p'_{j,k} = \frac {1- p_{0,k} }{p_{0,k}}$. The random variables $\sigma_k$ satisfy the following recursive distributional relationship
\begin{align*}
\sigma_1 &\overset{d}= 1 + \sigma_2,\\
\sigma_k &\overset{d}=
\begin{cases}
      \Geo(1-p_{0,k})+\sum \limits_{j=k+1}^{2k} p'_{j,k} \sigma_{k+j}, & 2 \leq k\leq \frac{n}{2} \\
      \Geo(1-p_{0,k})+\sum \limits_{j=k+1}^{n} p'_{j,k} \sigma_{k+j}, &  \frac{n}{2} < k \leq n-1
   \end{cases}.
\end{align*}
Recall that $T_n = \sigma_1$.
\end{proposition}

\begin{proof}
The expression for $\sigma_1$ is the observation that after one step there will always be two frogs awake. When $k \geq 2$ frogs are awake the time to wake more frogs is a geometric random variable with mean $1- p_{0,k}$. Conditioned that the $k$ frogs wake another frog, we obtain $j$ more awake frogs with probability $p'_{j,k}$. In this situation we now must wait $\sigma_{k+j}$ steps.
\end{proof}

\section{Appendix}

\subsection*{Phases for argument in \cite{original}}
\begin{enumerate}[label = \roman*.]
    \item The time inform $N$ residents for some fixed constant $N$ (not growing with $n$).
    \item The time to go from $N$ to $\zeta n$ informed residents with $0<\zeta<1$ a fixed constant.
    \item The time to go from $\zeta n$ to $(1- \epsilon)n$ with $\epsilon >0$ a small fixed constant.
    \item The time to go from $(1-\epsilon)n$ to $n-R$ informed residents where $R$ is a large fixed constant.
    \item The time to go from $n-R$ to $n$ informed residents.
\end{enumerate}

\subsection*{Deterministic equation in \cite{second}}
Letting $N(t)$ be the number of informed residents at time $t$:
$$N(t+1) =n- (n-N(t))\exp(-N(t)/n ).$$

\subsection*{Phases for argument in \cite{third}}
\begin{enumerate}[label = \roman*.]
    \item The time inform $\sqrt n$ residents.
    \item The time to go from $\sqrt n $ to $n/2$ informed residents.
    \item The time to go from $n/2$ to $n$ informed residents.
\end{enumerate}

\bibliographystyle{amsalpha}
\bibliography{frog_cover}

\end{document}